\documentclass[a4paper]{amsart}

\PII{}
\makeatletter
\def\@setcopyright{\@empty}
\makeatother

\newcommand{\Si}[1]{(1-#1^2)}
\newcommand{\Co}[1]{\cos^4\frac{#1}2}
\newcommand{\allp}{1\le p\le\infty}
\newcommand{\Px}[1]{P_{#1}^{(2,2)}}
\newcommand{\E}{E_n(f)_{p,\alpha}}
\newcommand{\Epar}[2]{E_{#1}\left(#2\right)_{p,\alpha}}
\newcommand{\T}[3]{\tau_{#1}^{#2}\left(#3\right)}
\newcommand{\hatT}[3]{\hat\tau_{#1}^{#2}\left(#3\right)}

\newcommand{\w}{\hat\omega(f,\delta)_{p,\alpha}}
\newcommand{\wpar}[1]{\hat\omega\left(#1\right)_{p,\alpha}}
\newcommand{\norm}[1]{\left\|#1\right\|_{p,\alpha}}
\newcommand{\normpar}[2]{\left\|#1\right\|_{#2}}
\newcommand{\Py}[1]{P_{#1}^{(0,4)}}
\newcommand{\Lp}{L_{p,\alpha}}
\newcommand{\Copar}[2]{\left(\cos \frac{#1}2\right)^{4#2}}
\newcommand{\prn}[1]{\left(#1\right)}
\newcommand{\brc}[1]{\left\{#1\right\}}

\newcommand{\numericset}[1]{\mathbb #1}
\newcommand{\numN}{\numericset N}

\newtheorem{theorem}{Theorem}[section]
\newtheorem*{theorem A}{Theorem A}
\newtheorem*{theorem B}{N\"olker's Theorem}
\newtheorem{lemma}{Lemma}[section]

\theoremstyle{remark}

\theoremstyle{remark}

\theoremstyle{definition}

\numberwithin{equation}{section}
\def\({\left ( }
\def\){\right )}
\def\<{\left < }
\def\>{\right >}

\newcounter{const}
\numberwithin{const}{theorem}
\numberwithin{const}{lemma}

\makeatletter
\newcommand{\Cn}[1][]{%
	\stepcounter{const}C_{\theconst}%
  \@ifnotempty{#1}{%
    \newcounter{#1}\setcounter{#1}{\arabic{const}}}}
\makeatother

\newcommand{\lastC}{C_{\theconst}}
\newcommand{\prevC}[1][1]{%
	{\countdef\n=255
	 \n=\theconst
	 \advance\n by-#1
	 C_{\number\n}}}

\numberwithin{equation}{section}
\numberwithin{theorem}{section}

\renewcommand{\theconst}{\arabic{const}}

\newsavebox\boxcst
\newcommand\boxconst[1]{\sbox\boxcst{$\Cn\label{#1}$}}
\newcommand\boxedconst{\usebox{\boxcst}}

\begin{document}

\title[Approximating some classes of functions]{%
	Approximating classes of functions
    defined by a generalised modulus of smoothness%
}

\author{Faton M.~Berisha}
\address{F.~M.\ Berisha\\
	Faculty of Mathematics and Sciences\\
	University of Prishtina\\
	N\"ena Terez\"e~5\\
	10000 Prishtin\"e\\
	Kosovo%
}
\email{faton.berisha@uni-pr.edu}

\author{Nimete Sh.~Berisha}
\address{N.~Sh.\ Berisha\\
	Faculty of Mathematics and Sciences\\
	University of Prishtina\\
	N\"ena Terez\"e~5\\
	10000 Prishtin\"e\\
	Kosovo%
}
\email{nimete.berisha@gmail.com}

\date{}

\begin{abstract}
	In the present paper,
	we use a generalised shift operator
	in order to define a generalised modulus of smoothness.
	By its means,
	we define generalised Lipschitz classes of functions,
	and we give their constructive characteristics.
	Specifically,
	we prove certain direct and inverse types theorems
	in approximation theory
	for best approximation
	by algebraic polynomials.
\end{abstract}

\maketitle

\section{Introduction}

In~\cite{p-berisha:anal-99},
a generalised shift operator was introduced,
by its means the generalised modulus of smoothness
was defined,
and Jackson's and its converse type theorems
were proved for this modulus.

In the present paper,
we make use of this modulus of smoothness
to define generalised Lipschitz classes of functions.
We prove the coincidence
of such a generalised Lipshcitz class
wth the class of functions
having a given order of decrease
of best approximation by algebraic polynomials.

\section{Definitions}

By $L_p[a,b]$ we denote the set of functions~$f$
such that for $1\le p<\infty$
$f$~is a measurable function on the segment $[a,b]$
and
\begin{displaymath}
	\normpar f p
	=\biggl(\int_a^b|f(x)|^p\,dx\biggr)^{1/p}<\infty,
\end{displaymath}
and for $p=\infty$ the function~$f$ is continuos
on the segment $[a,b]$
and
\begin{displaymath}
	\normpar f\infty=\max_{a\le x\le b}|f(x)|.
\end{displaymath}
In case that $[a,b]=[-1,1]$ we simply write~$L_p$
instead of $L_p[-1,1]$.

Denote by~$\Lp$ the set of functions~$f$
such that
$f(x)\*\Si{x}^\alpha\in L_p$,
and put
\begin{displaymath}
	\norm f=\|f(x)\Si{x}^\alpha\|_p.
\end{displaymath}

Denote by $\E$ the best approximation
of a function $f\in\Lp$
by algebraic polynomials of degree not greater than $n-1$,
in $\Lp$ metrics,
i.e.,
\begin{displaymath}
	\E=\inf_{P_n}\norm{f-P_n},
\end{displaymath}
where~$P_n$ is an algebraic polynomial
of degree not greater than $n-1$.

By~$E(p,\alpha,\lambda)$
we denote the class of functions $f\in\Lp$
satisfying the condition
\begin{displaymath}
	\E\le Cn^{-\lambda},
\end{displaymath}
where $\lambda>0$
and~$C$ is a constant
not depending on~$n$ $(n\in\numN)$.

Define generalised shift operator
$\hatT t{}{f,x}$
by
\begin{displaymath}
	\hatT t{}{f,x}=\frac1 {\pi\Si x\Co t}
		\int_0^\pi B_{\cos t}(x,\cos\varphi,R)
			f(R)\,d\varphi,
\end{displaymath}
where
\begin{align}\label{eq:R-B}
	R 			&=x\cos t-\sqrt{1-x^2}\sin t\cos\varphi, \notag\\
	B_y(x,z,R) 	&=2\Bigl(\sqrt{1-x^2}y+xz\sqrt{1-y^2}\\
				& \quad
		+\sqrt{1-x^2}(1-y)\Si z\Bigr)^2-\Si R. \notag
\end{align}

For a function $f\in\Lp$,
define the generalised modulus of smoothness
by
\begin{displaymath}
	\w=\sup_{|t|\le\delta}
		\norm{\hatT t{}{f,x}-f(x)}.
\end{displaymath}

Consider the class $H(p,\alpha,\lambda)$
of functions $f\in\Lp$
satisfying the condition
\begin{displaymath}
	\w\le C\delta^\lambda,
\end{displaymath}
where $\lambda>0$
and~$C$ is a constant
not depending on~$\delta$.

Put $y=\cos t$, $z=\cos\varphi$ in the operator $\hatT t{}{f,x}$,
denote it by~$\T y{}{f,x}$
and rewrite it in the form
\begin{displaymath}
	\T y{}{f,x}=\frac4{\pi\Si x(1+y)^2}
		\int_{-1}^1 B_y(x,z,R)f(R)\frac{dz}{\sqrt{1-z^2}},
\end{displaymath}
where~$R$ and $B_y(x,z,R)$ are defined in~\eqref{eq:R-B}.

By~$P_\nu^{(\alpha,\beta)}(x)$ $(\nu=0,1,\dotsc)$
we denote the Jacobi polynomials, 
i.e., the algebraic polynomials of degree~$\nu$,
orthogonal with the weight function
$(1-x)^{\alpha}(1+x)^{\beta}$
on the segment~$[-1,1]$,
and normed by the condition
\begin{displaymath}
	P_\nu^{(\alpha,\beta)}(1)=1 \quad(\nu=0,1,\dotsc).
\end{displaymath}

Denote by~$a_n(f)$ the Fourier--Jacobi coefficients
of a function~$f$,
integrable with the weight function $\Si{x}^2$
on the segment~$[-1,1]$,
with respect to the system of Jacobi polynomials
$\brc{\Px n(x)}_{n=0}^\infty$,
i.e.,
\begin{displaymath}
	a_n(f)=\int_{-1}^1f(x)\Px n(x)\Si{x}^2\,dx\quad(n=0,1,\dotsc).
\end{displaymath}

\section{Auxiliary statements}

In order to prove our results
we need the following theorem.

\begin{theorem}\label{th:jackson}
	Let the numbers~$p$ and~$\alpha$ be such that $\allp$;
	\begin{alignat*}2
		1/2 			   &<\alpha\le1
		  &\quad &\text{for $p=1$},\\
		1-\frac1{2p} &<\alpha<\frac32-\frac1{2p}
		  &\quad &\text{for $1<p<\infty$},\\
		1 				   &\le\alpha<3/2
		  &\quad &\text{for $p=\infty$}.
	\end{alignat*}
	If $f\in\Lp$,
	then for every natural number~$n$
	\begin{displaymath}
		\Cn\E\le\wpar{f,1/n},
	\end{displaymath}
	where the positive constant~$\lastC$
	does not depend on~$f$ and~$n$.
\end{theorem}

Theorem~\ref{th:jackson}
was proved in~\cite{p-berisha:anal-99}
and,
in more general form,
in~\cite{p-berisha:fundam-99}.
It is known as a Jackson's type theorem.

We also need the following lemmas.

\begin{lemma}\label{lm:properties-tau}
	The operator~$\T y{}{f,x}$ has the following properties:
	\begin{enumerate}
	\item\label{it:properties-tau-1}
		it is linear,
	\item\label{it:properties-tau-2}
		$\T1{}{f,x}=f(x)$,
	\item\label{it:properties-tau-3}
		$\T y{}{\Px\nu,x}=\Px\nu(x)\Py\nu(y)
		\quad(\nu=0,1,\dotsc)$,
	\item\label{it:properties-tau-4}
		$\T y{}{1,x}=1$,
	\item\label{it:properties-tau-6}
		$a_n(\T y{}{f,x})=a_n(f)\Py n(y)
			\quad(n=0,1,\dotsc)$.
	\end{enumerate}
\end{lemma}

Lemma~\ref{lm:properties-tau}
was proved in~\cite{p-berisha:anal-99}

\begin{lemma}\label{lm:bound-tau}
	Let the numbers~$p$ and~$\alpha$ be such that $\allp$;
	\begin{alignat*}2
		1/2				&<\alpha\le1
		  &\quad &\text{for $p=1$},\\
		1-\frac1{2p}	&<\alpha<\frac32-\frac1{2p}
		  &\quad &\text{for $1<p<\infty$},\\
		1				&\le\alpha<3/2
		  &\quad &\text{for $p=\infty$}.
	\end{alignat*}
	If $f\in\Lp$,
	then
	\begin{displaymath}
		\norm{\hatT t{}{f,x}}\le\frac C{\Co t}\norm f,
	\end{displaymath}
	where constant~$C$ does not depend on~$f$ and~$t$.
\end{lemma}

Lemma~\ref{lm:bound-tau} was proved in~\cite{p-berisha:anal-99}.

\section{Statement of results}

\begin{theorem}\label{th:EsubH-tau}
	Let~$p$, $\alpha$ and~$\lambda$ be given numbers
	such that
	$\allp$;
	\begin{alignat*}2
		1-\frac1{2p}	&<\alpha<\frac32-\frac1{2p}
			&\quad &\text{for $1\le p<\infty$},\\
		1				&\le\alpha<\frac32
			&\quad &\text{for $p=\infty$}.
	\end{alignat*}
	and $0<\lambda<2$.
	Let $f\in\Lp$.
	If
	\begin{displaymath}
		\E\le Mn^{-\lambda},
	\end{displaymath}
	then
	\begin{displaymath}
		\w\le CM\delta^\lambda,
	\end{displaymath}
	where constant~$C$
	does not depend on~$f$, $M$ and~$\delta$.
\end{theorem}

\begin{proof}
	Let $P_n(x)$ be an algebraical polynomial
	of degree not greater than $n-1$
	such that
	\begin{displaymath}
		\norm{f-P_n}=\E \quad(n=1,2,\ldots).
	\end{displaymath}
	We define algebraical polynomials $Q_k(x)$ by
	\begin{displaymath}
		Q_k(x)=P_{2^k}(x)-P_{2^{k-1}}(x)
		\quad (k=1,2,\ldots)
	\end{displaymath}
	and $Q_0(x)=P_1(x)$.
	Since for $k\ge1$
	\begin{multline*}
		\norm{Q_k}
		=\norm{P_2^k-P_{2^{k-1}}}
		\le\norm{P_{2^k}-f}+\norm{f-P_{2^{k-1}}}\\
		=\Epar{2^k}f+\Epar{2^{k-1}}f,
	\end{multline*}
	then by the conditions of the theorem we have
	\begin{equation}\label{eq:Qk-tau}
		\norm{Q_k}\le\Cn M2^{-k\lambda}.
	\end{equation}
	
	Taking into consideration property~\ref{it:properties-tau-4}
	in Lemma~\ref{lm:properties-tau}
	of the operator~$\tau_y$,
	without loss of generality
	we may suppose that $t\ne0$.
	For $0<|t|\le\delta$
	we estimate
	\begin{displaymath}
		I=\norm{\hatT t{}{f,x}-f(x)}.
	\end{displaymath}
	For every positive integer~$N$,
	taking into account property~\ref{it:properties-tau-1}
	in Lemma~\ref{lm:properties-tau}
	and the linearity of the operator $\T t{}{f,x}$,
	we get
	\begin{displaymath}
		I\le\norm{\hatT t{}{f-P_{2^N},x}-(f(x)-P_{2^N}(x))}
			+\norm{\hatT t{}{P_{2^N},x}-P_{2^N}(x)}.
	\end{displaymath}
	Since
	\begin{displaymath}
		P_{2^N}(x)=\sum_{k=0}^N Q_k(x),
	\end{displaymath}
	we have
	\begin{multline*}
		I\le\norm{\hatT t{}{f-P_{2^N},x}-(f(x)-P_{2^N}(x))}
			+\sum_{k=0}^N\norm{\hatT t{}{Q_k,x}-Q_k(x))}\\
		=J+\sum_{k=1}^N I_k.
	\end{multline*}
	
	Let~$N$ be chosen in such a way that
	\begin{equation}\label{eq:delta-tau}
		\frac\pi{2^N}<\delta\le\frac\pi{2^{N-1}}.
	\end{equation}
	We prove the following inequalities
	\boxconst{cn:J-tau}
	\begin{equation}\label{eq:J-tau}
		J\le\boxedconst M\delta^\lambda
	\end{equation}
	and
	\begin{equation}\label{eq:Ik-tau}
		I_k\le\Cn M2^{-k\lambda},
	\end{equation}
	where constants~$\boxedconst$ and~$\lastC$
	do not depend on~$f$, $M$, $\delta$ and~$k$.
	
	First we consider~$J$.
	By Lemma~\ref{lm:bound-tau},
	taking into account that $|t|\le\delta$,
	we have
	\begin{multline*}
		\norm{\hatT t{}{f-P_{2^N},x}-(f(x)-P_{2^N}(x))}
		\le\frac\Cn{\Copar{t}{}}\norm{f-P_{2^N}}\\
		=\Cn\Epar{2^N}f
	\end{multline*}
	Therefore, the condition of the theorem
	and inequality~\eqref{eq:delta-tau}
	yield
	\begin{displaymath}
		\norm{\hatT t{}{f-P_{2^N},x}-(f(x)-P_{2^N}(x))}
		\le\Cn M2^{-N\lambda}
		\le\Cn M\delta^\lambda,
	\end{displaymath}
	which proves inequality~\eqref{eq:J-tau}.
	
	Now we prove inequality~\eqref{eq:Ik-tau}.
	Note that, taking into consideration Lemma~\ref{lm:bound-tau},
	we have
	\begin{displaymath}
		\norm{\hatT t{}{Q_k}}
		\le\frac{\Cn}{\Copar{t}{}}\norm{Q_k}.
	\end{displaymath}
	Hence,
	\begin{displaymath}
		I_k
		\le\frac{\Cn}{\Copar{t}{}}M2^{-k\lambda},
	\end{displaymath}
	which proves inequality~\eqref{eq:Ik-tau}.
	
	Inequalities~\eqref{eq:J-tau}, \eqref{eq:Ik-tau}
	and~\eqref{eq:delta-tau}
	yield
	\begin{displaymath}
		I\le\Cn M\prn{\delta^\lambda
			+\sum_{k=1}^N 2^{-k\lambda}}
		\le\Cn M(
			\delta^\lambda
			+2^{-N\lambda}
		)\\
		\le\Cn M\delta^\lambda.
	\end{displaymath}
	
	Theorem~\ref{th:EsubH-tau} is proved.
\end{proof}

\begin{theorem}\label{th:HsubE-tau}
	Let~$p$, $\alpha$ and~$\lambda$ be given numbers
	such that $\allp$, $\lambda>0$;
	\begin{alignat*}2
		1-\frac1{2p}	&<\alpha<\frac32-\frac1{2p}
			&\quad \text{for $1\le p<\infty$},\\
		1				&\le\alpha<\frac32 
			&\quad \text{for $p=\infty$}.
	\end{alignat*}
	Let $f\in\Lp$.
	If
	\begin{displaymath}
		\w\le M\delta^\lambda,
	\end{displaymath}
	then
	\begin{displaymath}
		\E\le CMn^{-\lambda},
	\end{displaymath}
	where constant~$C$
	does not depend on~$f$, $M$ and~$n$.
\end{theorem}

\begin{proof}
	Let $\delta=\frac1n$.
	Then,
	taking into account Theorem~\ref{th:jackson},
	we obtain
	\begin{displaymath}
		\E\le\frac1{\Cn}\wpar{f,\frac1n}
		\le CMn^{-\lambda}.
	\end{displaymath}
	
	Theorem~\ref{th:HsubE-tau} is proved.
\end{proof}

\begin{theorem}\label{th:coincidence-tau}
	Let~$p$, $\alpha$ and~$\lambda$ be given numbers
	such that
	$\allp$;
	\begin{alignat*}2
		1-\frac1{2p}	&<\alpha<\frac32-\frac1{2p}
			&\quad &\text{for $1\le p<\infty$},\\
		1				&\le\alpha<\frac32 
			&\quad &\text{for $p=\infty$}.
	\end{alignat*}
	Then for $0<\lambda<2$
	the classes of functions $H(p,\alpha,\lambda)$
	coincide with the class $E(p,\alpha,\lambda)$.
\end{theorem}

\begin{proof}
	Note that, under the condition of the theorem,
	Theorem~\ref{th:HsubE-tau} implies the inclusion
	\begin{displaymath}
		H(p,\alpha,\lambda)\subseteq E(p,\alpha,\lambda),
	\end{displaymath}
	while Theorem~\ref{th:EsubH-tau} implies the converse inclusion
	\begin{displaymath}
		E(p,\alpha,\lambda)\subseteq H(p,\alpha,\lambda).
	\end{displaymath}
	Hence we conclude
	that the assertion of Theorem~\ref{th:coincidence-tau}
	is implied by Theorems~\ref{th:HsubE-tau}
	and~\ref{th:EsubH-tau}.
\end{proof}

Note that analogues of Theorems~\ref{th:HsubE-tau},
\ref{th:EsubH-tau} and~\ref{th:coincidence-tau}
for another generalised shift operator
were proved in~\cite{berisha:math-98}
and, in more general forms,
in~\cite{p-berisha:east-98,b-berisha:anlysis-12}.

\bibliographystyle{amsplain}
\bibliography{maths}

\providecommand\cprime{'}
\providecommand{\bysame}{\leavevmode\hbox to3em{\hrulefill}\thinspace}
\providecommand{\MR}{\relax\ifhmode\unskip\space\fi MR }
\providecommand{\MRhref}[2]{%
  \href{http://www.ams.org/mathscinet-getitem?mr=#1}{#2}
}
\providecommand{\href}[2]{#2}
\begin{thebibliography}{1}

\bibitem{berisha:math-98}
F.~M. Berisha, \emph{On coincidence of classes of functions defined by a
  generalized modulus of smoothness and the appropriate inverse theorem}, Math.
  Montisnigri \textbf{9} (1998), 15--36. \MR{1657672 (99i:41021)}

\bibitem{b-berisha:anlysis-12}
N.~Sh. Berisha and F.~M. Berisha, \emph{Approximating classes of functions
  defined by operators of differentiation or operators of generalised
  translation by means of algebraic polynomials}, Int. J. Math. Anal. (Ruse)
  \textbf{6} (2012), no.~55, 2709--2729.

\bibitem{p-berisha:east-98}
M.~K. Potapov and F.~M. Berisha, \emph{Approximation of classes of functions
  defined by a generalized $k$-th modulus of smoothness}, East J. Approx.
  \textbf{4} (1998), no.~2, 217--241. \MR{1638345 (99g:41007)}

\bibitem{p-berisha:anal-99}
\bysame, \emph{Direct and inverse theorems of approximation theory for a
  generalized modulus of smoothness}, Anal. Math. \textbf{25} (1999), no.~3,
  187--203. \MR{1717821 (2000h:41009)}

\bibitem{p-berisha:fundam-99}
\bysame, \emph{O svyazi mezhdu~$r$-im obobshchennym medulem gladkosti i
  nai\-lu\-ch\-shi\-mi priblizheniyami algebraicheskimi mnogochlenami}, Fundam.
  Prikl. Mat. \textbf{5} (1999), no.~2, 563--587. \MR{1803600 (2001j:41004)}

\end{thebibliography}

\end{document}